\documentclass[10pt]{article}
\usepackage{bbm}
\usepackage{mathrsfs}
\usepackage{amsfonts}
\usepackage{amsthm,amsmath,amssymb,anysize}
\newtheorem{lemma}{Lemma}[section]
\newtheorem{theorem}[lemma]{Theorem}
\newtheorem{remark}[lemma]{Remark}
\newtheorem{proposition}[lemma]{Proposition}

\usepackage[numbers,sort&compress]{natbib}
\marginsize{3.2cm}{3.2cm}{3.6cm}{3cm}
\numberwithin{equation}{section}

\title{\textsf{Cohomology of $\frak{q}(2)$ in prime characteristic}\footnote{This work is supported by Heilongjiang Provincial Natural Science Foundation of China (YQ2020A005) and the Natural Science Foundation  of China (12061029).}}
\author{\textsc{Shujuan Wang$^{1}$, \textsc{Yang Liu$^{2}$} and
  \textsc{Wende Liu$^{3,}$}}\footnote{Correspondence:  wendeliu@ustc.edu.cn (W. Liu)}\\
  {\small \textit{$^1$Department of Mathematics, Shanghai Maritime University,}}\\
\small \textit{Shanghai  201306, China} \\
\small \textit{$^2$School of Mathematics, Harbin Institute of Technology,}\\
\small \textit{Harbin 150001, China} \\
 \small\textit{$^3$School of Mathematics and Statistics,}
\textit{Hainan Normal University,}\\ \small\textit{ Haikou 571158,  China} }
\date{ }
\begin{document}
\maketitle
\begin{quotation}
\small\noindent \textbf{Abstract}:
Over an algebraically closed field of characteristic $p>2$,
 the 0-dimensional and 1-dimensional cohomology of  the queer Lie superalgebra $\frak{q}(2)$ with
coefficients in all  baby Verma modules and all the simple modules are determined.

 \vspace{0.2cm} \noindent{\textbf{Keywords}}: Queer Lie superalgebras, baby Verma modules, simple modules, cohomology.

\vspace{0.1cm} \noindent \textbf{Mathematics Subject Classification
2020}: 17B40, 17B50, 17B56

\end{quotation}
\setcounter{section}{0}
\section{Introduction}
 Lie superalgebra cohomology is of great importance for studying extensions of
modules as well as extensions of Lie superalgebras themselves.
For instance, relative
cohomology is fundamental in the Borel-Weil-Bott theory (see \cite{BWB}) and cohomology of
nilpotent radicals of parabolic subalgebras is crucial in the Kazhdan-Lusztig theory (see \cite{KL}).
In 1977, Kac posed a challenging question: determining low-dimensional
cohomology of simple Lie superalgebras with coefficients in arbitrary simple modules over a field of
characteristic 0.
For the simple Lie superalgebra $\frak{sl}(m,n)$ or $\frak{osp}(2,2n)$,
Schunert, Su and Zhang have answered  Kac's question in \cite[p. 5052]{schunertzhang} and \cite[Theorems 1.2 and 1.3]{suzhang}.

Our concern is the modular-version of Kac's question above.
This paper is a sequel to \cite{sl(21)}, in which $\mathrm{H}^1\left(\frak{sl}(2,1), M\right)$
is determined for any finite-dimensional simple $\frak{sl}(2,1)$-module $M$ over a field of prime characteristic.
This paper aims to determine the 0-dimensional and 1-dimensional cohomology of  the smallest queer Lie superalgebra $\frak{q}(2)$ with
coefficients in all  baby Verma modules and all simple modules
over an algebraically closed field of characteristic $p>2$.
Our main results are the following two theorems.
\begin{theorem}\label{Th1}
Over an algebraically closed field of characteristic $p>2$, let $ Z_{\chi}(\lambda)$ be the baby Verma module of $q(2)$ with   highest weight $\lambda$ and $p$-character $\chi$.
 Then
\begin{align}
\mathrm{sdim}\; \mathrm{H}^0(q(2),  Z_{\chi}(\lambda))&=\left\{\begin{array}{lll}
0\mid1, &(\lambda,\chi)=(0,0)\\
0\mid0,\;\;&\mbox{otherwise},
\end{array}\right.\label{0Z}\\
\mathrm{sdim}\; \mathrm{H}^1(q(2),  Z_{\chi}(\lambda))&=\left\{\begin{array}{lll}
1\mid1, &(\lambda,\chi)=(0,0)\\
0\mid0,\;\;&\mbox{otherwise}.
\end{array}\right.\label{1Z}
\end{align}
\end{theorem}
\begin{theorem}\label{Th2}
Over an algebraically closed field of characteristic $p>2$, let $L_{\chi}(\lambda)$ be the simple module of $q(2)$ with   highest weight $\lambda$ and $p$-character $\chi$.
 Then
\begin{align*}
\mathrm{sdim}\; \mathrm{H}^1(q(2),L_{\chi}(\lambda))&=\left\{\begin{array}{ll}
2\mid 2, &(\lambda,\chi)=(0,0)\\
0\mid 1, &\lambda=(1,p-1),\chi=0\\
 2\mid 0, &\lambda=(p-1,1),\chi=0\\
0\mid 0, & \mbox{otherwise}.
\end{array}\right.
\end{align*}
\end{theorem}

\section{Preliminaries}
In this paper,  we write $\mathbb{F}$ for the underlying field,
 $\mathbb{F}_p$ for the prime subfield of $\mathbb{F}$.
 All vector spaces, algebras  and (sub)modules are assumed to be $\mathbb{Z}_2$-graded and finite-dimensional over $\mathbb{F}$.
Hereafter $\mathbb{Z}_2=\{\bar{0}, \bar{1}\}$ is the field of two elements.
 We make some convention for a vector (super)space $V=V_{\bar{0}}\oplus V_{\bar{1}}$:
 \begin{itemize}
 \item[(1)]
 For $v\in V_{\bar{0}}\cup V_{\bar{1}}$, write $|v|\in\mathbb{Z}_2$ for the parity ($\mathbb{Z}_2$-degree) of $v$ and
 the symbol $|a|$ always implies that $a$ is $\mathbb{Z}_2$-homogeneous in a vector space.
 \item[(2)]
Write $\mathrm{sdim}\;V=\dim V_{\bar{0}}\mid\dim V_{\bar{1}}$.
  \item[(3)]
Write
$V=\langle v_1, \ldots, v_m\mid w_1, \ldots, w_n\rangle,$
meaning that $\{v_1, \ldots, v_m\mid w_1, \ldots, w_n\}$ is a $\mathbb{Z}_2$-homogeneous basis of $V$.
In case $m=0$ or  $n=0$, write
$V=\langle0\mid w_1, \ldots, w_n\rangle$ or $\langle v_1, \ldots, v_m\mid 0\rangle$, respectively.
 \end{itemize}
\subsection{The low-dimensional cohomology of a Lie superalgebra}
Let $L$ be a Lie superalgebra and $M$ an $L$-module.
Recall that a $\mathbb{Z}_{2}$-homogeneous linear mapping $\varphi: L\longrightarrow M$
is a \textit{derivation of parity $|\varphi|$},  if
\begin{eqnarray*}\label{low00}
\varphi([x,y])=(-1)^{|\varphi| |x|}x \varphi(y)-(-1)^{|y|
(|\varphi|+|x|)}y \varphi(x)\; \mbox{for}\; x, y\in L.
\end{eqnarray*}
Denote by $\mathrm{Der}(L,M)$ the vector space spanned by all the $\mathbb{Z}_{2}$-homogeneous derivations from $L$ to $M$,
 each element in which is called a \textit{derivation}.
 For a $\mathbb{Z}_{2}$-homogeneous element $m\in M$,   the map $\frak{D}_m$ from $L$ to $M$ is defined by
 $$\frak{D}_m(x)=(-1)^{|x||m|}x m,\; \mbox{where}\; x\in L.$$
 Then $\frak{D}_m$ is a $\mathbb{Z}_{2}$-homogeneous derivation of parity $|m|$.
 Write $\mathrm{Ider}(L,M)$ for the vector space spanned by all $\frak{D}_m$ with $\mathbb{Z}_{2}$-homogeneous elements  $m\in M$,
 each element in which is called an \textit{inner derivation}.
 In general,  $L$-module $\mathrm{Hom}_{\mathbb{F}}(L,M)$ (consisting of all linear maps from $L$  to $M$) contains
 $\mathrm{Ider}(L,M)$ and $\mathrm{Der}(L,M)$ as submodules.

 Let $\frak{h}$ be a Cartan subalgebra of $L$.
 Suppose that $L$ and  $M$ possess  weight space decompositions with respect to $\frak{h}_{\bar{0}}$:
$$L=\oplus _{\gamma \in \frak{h}_{\bar{0}}^*}L_{\gamma}, \quad M=\oplus _{\gamma \in
\frak{h}_{\bar{0}}^*}M_{\gamma}.$$
Hereafter, denote by $L^*$ the space consisting of all linear maps from $L$  to $\mathbb{F}$ for Lie superalgebra $L$.
Write
\begin{align*}
\mathrm{Hom}_{\mathbb{F}}(L,M)_{(0)}&=\{\phi\in \mathrm{Hom}_{\mathbb{F}}(L,M)\mid \phi(L_{\alpha})\subset M_{\alpha}, \forall \alpha \in  \frak{h}_{\bar{0}}^*\},\\
\mathrm{Der}(L,M)_{(0)}&=\{\phi\in \mathrm{Der}(L,M)\mid \phi(L_{\alpha})\subset M_{\alpha}, \forall \alpha \in  \frak{h}_{\bar{0}}^*\}.
\end{align*}
A linear map  (resp. derivation) in $\mathrm{Hom}_{\mathbb{F}}(L,M)_{(0)}$ (resp. $\mathrm{Der}(L,M)_{(0)}$)  is called a \textit{weight-map} (resp. \textit{weight-derivation}) with respect to $\frak{h}$.
It is a standard fact that
\begin{align}\label{2006131342}
\mathrm{Der}(L,M)=\mathrm{Der}(L,M)_{(0)}+\mathrm{Ider}(L,M),
\end{align}
for a standard proof of which the reader can see \cite[Lemma 3.2]{Bai} or \cite[Lemma 2.1]{sl(21)}.

By definition, the
1-dimensional cohomology of $L$ with coefficients in $M$ is
\begin{align}\label{2006131408}
\mathrm{H}^1(L,M)=\mathrm{Der}(L,M)/\mathrm{Ider}(L,M);
\end{align}
and the 0-dimensional cohomology   is
\begin{align*}
\mathrm{H}^0(L,M)=\left\{m\in M\mid xm=0, \forall x\in L\right\}.
\end{align*}
Two $1$-cocycles (elements in $\mathrm{Der}(L,M)$) are said to be \textit{cohomologous} if their images in $H^1(L,M)$ are equal.
From (\ref{2006131342}) and (\ref{2006131408}), we get the following lemma,
which gives a useful  reduction method in computing the 1-dimensional cohomology of Lie superalgebras.
\begin{lemma}\label{reduction}
Retain the above notations.
Let $\varphi\in\mathrm{H}^1(L,M)$.
Then $\varphi$ is cohomologous to a weight-derivation.
In particular, $\varphi(h)$ lies in $\mathrm{H}^0(L,M)$ for any $h\in \frak{h}_{\bar{0}}$.
\end{lemma}
\begin{proof}
It is sufficient to show the last assertion.
Since $\varphi$ is cohomologous to a weight-derivation,
we may view $\varphi$ as a weight-derivation.
Let $h\in \frak{h}_{\bar{0}}$ and $x\in L_{\alpha}$ for any $\alpha\in \frak{h}_{\bar{0}}^*$.
Then
\begin{align*}
\alpha(h)\varphi(x)=\varphi([h,x])&=h\varphi(x)-(-1)^{|x||\varphi|}x\varphi(h)\\
&=\alpha(h)\varphi(x)-(-1)^{|x||\varphi|}x\varphi(h).
\end{align*}
It follows that $x\varphi(h)=0$.
This implies $\varphi(h)\in \mathrm{H}^0(L,M)$
since both $x$ and $\alpha$ are arbitrary.
\end{proof}

\subsection{The queer Lie superalgebra $\frak{q}(2)$ and its representation theory}\label{sec.21l}

We follow the reference \cite{wangzhao} for the structrue and representation theory of $q(2)$.
For the convenience of the readers, we summarize some information
as below.

For $k=1,2$, set $\dot{k}=2+k$ for convenience.
Write
 \begin{align*}
 h_1&:=E_{11}+E_{\dot{1}\dot{1}},\quad \;h_2:=E_{22}+E_{\dot{2}\dot{2}},\quad \;e:=E_{12}+E_{\dot{1}\dot{2}},\quad \;f:=E_{21}+E_{\dot{2}\dot{1}},\\ H_1&:=E_{1\dot{1}}+E_{\dot{1}1},\quad H_2:=E_{2\dot{2}}+E_{\dot{2}2},\quad E:=E_{1\dot{2}}+E_{\dot{1}2},\quad F:=E_{2\dot{1}}+E_{\dot{2}1}.
 \end{align*}
Hereafter $E_{ij}$  is the $4\times 4$ matrix unit.
The queer Lie superalgebra
$$q(2)=\langle h_1,h_2,e,f\mid H_1,H_2,E,F\rangle $$
and write it to $\frak{g}$ for short.
We call $\frak{h}:=\langle h_1,h_2\mid H_1,H_2\rangle$   the standard Cartan subalgebra of $\frak{g}$.
Let $\lambda\in\frak{h}_{\bar{0}}^*$.
If $\lambda(h_i)=\lambda_i$ for $i=1,2$,  write $\lambda=(\lambda_1,\lambda_2)$ for short.
With respect to $\frak{h}_{\bar{0}}$, all weight spaces of $\frak{g}$ are listed in the following
\begin{align*}
\frak{g}_{0}=\langle h_1,h_2\mid H_1,H_2\rangle,\quad
 \frak{g}_{(1,-1)}=\langle e\mid E\rangle, \quad
\frak{g}_{(-1,1)}=\langle f\mid F\rangle.
\end{align*}
Letting $\frak{n}^-=\langle f\mid F\rangle$ and $\frak{n}^+=\langle e\mid E\rangle$,
we have a triangular decomposition $\frak{g}=\frak{n}^-\oplus\frak{h}\oplus \frak{n}^+$.

Recall that a restricted Lie superalgebra is a Lie superalgebra, whose even part
is a restricted Lie algebra and
the odd part is a restricted module of the even part by the adjoint action.
Then
$\frak{g}$ is a restricted Lie superalgebra with a $p$-mapping $[p]$ being the usual $p$th power.
Let $V$ be a simple $\frak{g}$-module.
Then there exists $\chi\in \frak{g}_{\bar{0}}^*$,
such that
$$x^pv-x^{[p]}v=\chi(x)^pv, \quad \forall x\in \frak{g}_{\bar{0}},\;\; v\in V.$$
In this case we also call $V$  a simple $\frak{g}$-module with  $p$-character $\chi$.
Fix $\chi\in \frak{g}^*_{\bar{0}}$.
Denote by $I_{\chi}$  the ideal of $U(\frak{g})$ generated by the elements
$x^p-x^{[p]}-\chi(x)^p$
for all $x\in \frak{g}_{\bar{0}}$.
Write $U_{\chi}(\frak{g})=U(\frak{g})/I_{\chi}$,
which is called the reduced enveloping superalgebra with $p$-character $\chi$.
Note that a simple $\frak{g}$-module with $p$-character $\chi$
is the same as a simple $U_{\chi}(\frak{g})$-module.
Any $p$-character $\chi'$ is $G$-conjugate to a $p$-character $\chi$ with $\chi(\frak{n}_{\bar{0}}^+)=0$
and $U_{\chi'}(\frak{g})=U_{\chi}(\frak{g})$, where $\frak{g}_{\bar{0}} = \mathrm{Lie}(G).$
Therefore the study of simple  $\frak{g}$-modules
is reduced to a problem of studying simple ones  with  $p$-character $\chi$  when $\chi$ runs over the
representatives of coadjoint $G$-orbits in $\frak{g}_{\bar{0}}^*$ (see also \cite[Remark 2.3]{14}).
Recall that there are three coadjoint $G$-orbits
with the following representatives (see \cite[Sec. 6]{wangzhao}):
\begin{itemize}
\item[(1)]  nilpotent: $\chi(e)=\chi(h_1)=\chi(h_2)=0$ and $\chi(f)=1.$
\item[(2)] semisimple: $\chi(e)=\chi(f)=0, \chi(h_1)=a, \chi(h_2)=b$ for some $a,b\in \mathbb{F}$.
\item[(3)] mixed: $\chi(e)=0, \chi(f)=1, \chi(h_1)=\chi(h_2)=a$ for some $a\in \mathbb{F}\backslash\{0\}$.
\end{itemize}
Hereafter the symbol $\chi$ implies that
$\chi\in \frak{g}^*_{\bar{0}}$ and
$\chi$ is either  nilpotent or semisimple or mixed.

Below we recall simple $U_{\chi}(\frak{h})$-modules constructed in \cite[Sec. 2.3]{wangzhao}.
 Let $\lambda\in \frak{h}_{\bar{0}}^*$ and
 $\frak{h}_1$  a maximal isotropic subspace with respect to the following bilinear form on $\frak{h}_{\bar{1}}$
$$(a,b)_{\lambda}:=\lambda([a,b]),\quad \forall a,\;b \in \frak{h}_{\bar{1}}.$$
Then $\lambda$ can be extended to a one-dimensional $(\frak{h}_{\bar{0}}+\frak{h}_1)$-module $\mathbb{F}_{\lambda}$ of
by letting $\frak{h}_1\mathbb{F}_{\lambda}=0$.
Write
$$\Lambda_{\chi}=\left\{\lambda\in \frak{h}_{\bar{0}}^*\mid \lambda(h)^p-\lambda(h)=\chi(h)^p, h\in\frak{h}_{\bar{0}}\right\}.$$
Then $\mathbb{F}_{\lambda}$ is a $U_{\chi}(\frak{h}_{\bar{0}}+\frak{h}_1)$-module
 if and only if $\lambda\in \Lambda_{\chi}$.
Write
$$V_{\chi}(\lambda)=U_{\chi}(\frak{h})\otimes_{U_{\chi}(\frak{h}_{\bar{0}}+\frak{h}_1)}\mathbb{F}_{\lambda}, \quad \lambda\in \Lambda_{\chi}.$$
Then $V_{\chi}(\lambda)$ is a simple $U_{\chi}(\frak{h})$-module.
Recall that the baby Verma module of $U_{\chi}(\frak{g})$  with   highest weight $\lambda$ and $p$-character $\chi$
is
$$ Z_{\chi}(\lambda):=U_{\chi}(\frak{g})\otimes_{U_{\chi}(\frak{h}\oplus\frak{n}^+)}V_{\chi}(\lambda),\quad \lambda\in \Lambda_{\chi}.$$
Denote by $L_{\chi}(\lambda)$ the unique simple quotient of $ Z_{\chi}(\lambda)$,
which is also of highest weight $\lambda$ and $p$-character $\chi$.
Write $v_{\lambda}$ for the highest weight-vector of weight $\lambda$ in $V_{\chi}(\lambda)$ and set $|v_{\lambda}|=\bar{0}$.
For convenient, write
$( i, j, k)$ and $[i, j, k]$ for the elements $f^iF^jH_1^kv_{\lambda}$ and
$f^iF^jH_2^kv_{\lambda}$ in $ Z_{\chi}(\lambda)$, respectively.
In this paper, the symbols $f^a, (a,j,k), [a, j, k]$ and $\underline{a}$ always imply that $a$ is the smallest nonnegative integer
in the residue class containing $a$ modulo $p$.
We also use $( i, j, k)$ or $[i, j, k]$
to represent the residue class containing $( i, j, k)$ or $[i, j, k]$ in $L_{\chi}(\lambda)$.
If $x=\sum^{p-1}_{i=0}\sum_{j=0}^1\sum_{k=0}^1a_{ijk}(i,j,k)$,
we write $x^{(i,j,k)}$ for $a_{ijk}$.

\begin{remark}\cite[Sec. 5]{wangzhao}\label{basis}
Let $\chi\in\frak{g}_{\bar{0}}^*$ and $\lambda=(\lambda_1,\lambda_2)\in \Lambda_{\chi}$.
\begin{itemize}
\item[(1)]  If $\lambda=0$, $ Z_{\chi}(\lambda)$ has a basis
$$\left\{(a,j,0)\mid j=0,1, 0\leq a\leq p-1\right\}.$$

\item[(2)]  If $\lambda_1\neq 0$, $ Z_{\chi}(\lambda)$ has a basis
$$\left\{(a,j,k)\mid j,k=0,1, 0\leq a\leq p-1\right\}.$$

\item[(3)]  If $\lambda_1=0$ and $\lambda_2\neq0$, $ Z_{\chi}(\lambda)$ has a basis
$$\left\{[a,j,k]\mid j,k=0,1, 0\leq a\leq p-1\right\}.$$
\end{itemize}
\end{remark}
\section{$\mathrm{H}^0(\frak{g},  Z_{\chi}(\lambda))$}
In this paper, the symbol $\delta_{P}$ means 1 if a proposition $P$ is true, and 0 otherwise.
We list some formulas about the $\frak{g}$-action on $ Z_{\chi}(\lambda)$
(see \cite[Sec.5]{wangzhao} for details) in the following.

\begin{remark}\label{200823907}
The $\frak{g}$-action on $ Z_{\chi}(\lambda)$ is given in the following.
\begin{itemize}
\item[(1)]
Let $\lambda=(0,\lambda_2)\neq 0$. Then
\begin{align}
F[a,j,k]=&\delta_{j=0}[a,1,k]\nonumber,\\
f[a,j,k]=&\delta_{a\neq p-1}[a+1,j,k]+\delta_{a= p-1}\chi(f)^p[0,j,k]\nonumber,\\
H_i[a,0,k]=&(-1)^{i}a[a-1,1,k]+\delta_{i=2}\left(\delta_{k=0}+\lambda_2\delta_{k=1}\right)[a,0,\delta_{k=0}]\nonumber,\\
E[a,1,k]=&a\left(\delta_{k=0}+\lambda_2\delta_{k=1}\right)[a-1,1,\delta_{k=0}]+\lambda_2[a,0,k]\label{E002},\\
E[a,0,k]=&-a\left(\delta_{k=0}+\lambda_2\delta_{k=1}\right)[a-1,0,\delta_{k=0}]-a(a-1)[a-2,1,k]\nonumber,\\
e[a,j,k]=&-a\left(a-(-1)^j+\lambda_2\right)[a-1,j,k]\nonumber\\
&-\delta_{(j,k)=(1,0)}[a,0,1]-\delta_{(j,k)=(1,1)}\lambda_2[a,0,0]\nonumber,\\
H_i[a,1,k]=&\delta_{a\neq p-1}[a+1,0,k]+\delta_{a= p-1}\chi(f)^p[0,0,k]\nonumber\\
&-\delta_{i=2}\left(\delta_{k=0}+\lambda_2\delta_{k=1}\right)[a,1,\delta_{k=0}].\nonumber
\end{align}

\item[(2)]
If $\lambda=(\lambda_1,0)\neq 0$, then
\begin{align*}
F(a,j,k)=&\delta_{j=0}(a,1,k),\\
f(a,j,k)=&\delta_{a\neq p-1}(a+1,j,k)+\delta_{a= p-1}\chi(f)^p(0,j,k),\\
H_i(a,0,k)=&(-1)^{i}a(a-1,1,k)+\delta_{i=1}\left(\delta_{k=0}+\lambda_1\delta_{k=1}\right)(a,0,\delta_{k=0}),\\
E(a,1,k)=&-a\left(\delta_{k=0}+\lambda_1\delta_{k=1}\right)(a-1,1,\delta_{k=0})+\lambda_1(a,0,k),\\
E(a,0,k)=&a\left(\delta_{k=0}+\lambda_1\delta_{k=1}\right)(a-1,0,\delta_{k=0})-a(a-1)(a-2,1,k),\\
e(a,j,k)=&a\left(\lambda_1-a+(-1)^j\right)(a-1,j,k)\\
&+\delta_{(j,k)=(1,0)}(a,0,1)+\delta_{(j,k)=(1,1)}\lambda_1(a,0,0),
\end{align*}

\begin{equation}
\begin{split}\label{H1010}
H_i(a,1,k)=&\delta_{a\neq p-1}(a+1,0,k)+\delta_{a= p-1}\chi(f)^p(0,0,k)\\
&-\delta_{i=1}\left(\delta_{k=0}+\lambda_1\delta_{k=1}\right)(a,1,\delta_{k=0}).
\end{split}
\end{equation}
\item[(3)]
If $\lambda_1=\lambda_2\neq0$  or
 $0\neq \lambda_1^2\neq \lambda_2^2\neq 0$,
then
\begin{align*}
F(a,j,k)=&\delta_{j=0}(a,1,k),\\
f(a,j,k)=&\delta_{a\neq p-1}(a+1,j,k)+\delta_{a= p-1}\chi(f)^p(0,j,k),\\
E(a,1,k)=&-a\left(\delta_{k=0}(1+\mu^{-1})+(\lambda_1+\mu\lambda_2)\delta_{k=1}\right)(a-1,1,\delta_{k=0})\\
&+(\lambda_1+\lambda_2)(a,0,k),\\
E(a,0,k)=&a\left(\delta_{k=0}(1+\mu^{-1})+(\lambda_1+\mu\lambda_2)\delta_{k=1}\right)(a-1,0,\delta_{k=0})\\
&-a(a-1)(a-2,1,k),\\
e(a,j,k)=&\delta_{(j,k)=(1,1)}(\lambda_1+\mu\lambda_2)(a,0,0)+\delta_{(j,k)=(1,0)}(1+\mu^{-1})(a,0,1)\\
&+a\left(\lambda_1-\lambda_2-a+(-1)^j\right)(a-1,j,k),\\
H_i(a,0,k)=&(-1)^{i}a(a-1,1,k)+\delta_{i=1}\left(\delta_{k=0}+\lambda_1\delta_{k=1}\right)(a,0,\delta_{k=0})\\
&-\delta_{i=2}\left(\delta_{k=0}\mu^{-1}+\mu\lambda_2\delta_{k=1}\right)(a,0,\delta_{k=0}),
\end{align*}
\begin{equation}
\begin{split}\label{H1111}
H_i(a,1,k)=&\delta_{a\neq p-1}(a+1,0,k)+\delta_{a= p-1}\chi(f)^p(0,0,k)\\
&-\delta_{i=1}\left(\delta_{k=0}+\lambda_1\delta_{k=1}\right)(a,1,\delta_{k=0})\\
&+\delta_{i=2}\left(\delta_{k=0}\mu^{-1}+\mu\lambda_2\delta_{k=1}\right)(a,1,\delta_{k=0}),
\end{split}
\end{equation}
where $\mu^2=-1$ if $\lambda_1=\lambda_2\neq0$, and $\mu\lambda_2+\lambda_1\mu^{-1}=0$ if
 $0\neq \lambda_1^2\neq \lambda_2^2\neq 0$.

\item[(4)] If $\lambda=(\lambda_1,-\lambda_1)$, then
\begin{align}
&h_i(a,j,k)=(-1)^i(a+j-\lambda_1)(a,j,k)\label{h},\\
&F(a,j,k)=\delta_{j=0}(a,1,k)\label{F},\\
&f(a,j,k)=\left\{\begin{array}{ll}
\chi(f)^p(0,j,k), &\mbox{if  $a= p-1$}\\
(a+1,j,k), &\mbox{if $a\neq p-1$}
\end {array}\right.\label{f}.
\end{align}
Besides, the following  formulas are true.
\begin{itemize}
\item If $\lambda_1\neq 0$, then
\begin{align}
&e(a,j,k)=a\left(2\lambda_1-a+(-1)^j\right)(a-1,j,k)+\delta_{(j,k)=(1,0)}2(a,0,1)\label{e},\\
&H_i(a,1,k)=\left((-1)^i\delta_{k=0}-\lambda_1\delta_{k=1}\right)(a,1,\delta_{k=0})+\left\{\begin{array}{ll}
\chi(f)^p(0,0,k), &\mbox{if  $a= p-1$}\\
(a+1,0,k), &\mbox{if $a\neq p-1$}
\end {array}\right.\label{a1k},\\
&H_i(a,0,k)=\left((-1)^{i+1}\delta_{k=0}+\lambda_1\delta_{k=1}\right)(a,0,\delta_{k=0})+(-1)^{i}a(a-1,1,k)\label{a0k},\\
&E(a,j,k)=\delta_{(j,k)\neq(1,1)}\left(\delta_{k=0}(-1)^j2a(a-1,j,1)-\delta_{j=0}a(a-1)(a-2,1,k)\right)\label{E}.
\end{align}
\item  If $\lambda_1= 0$, then
\begin{align}
&e(a,j,0)=-a\left(a-(-1)^j\right)(a-1,j,0)\label{e0},\\
&H_i(a,1,0)=\left\{\begin{array}{ll}
\chi(f)^p(0,0,0), &\mbox{if  $a= p-1$}\\
(a+1,0,0), &\mbox{if $a\neq p-1$}
\end {array}\right.\label{a10},\\
&H_i(a,0,0)=(-1)^{i}a(a-1,1,0)\label{a00},\\
&E(a,j,0)=-\delta_{j=0}a(a-1)(a-2,1,0)\nonumber.
\end{align}
\end{itemize}
\end{itemize}
\end{remark}

In the following, we give a proof of the first part of Theorem \ref{Th1}.

\textit{The proof of Formula (\ref{0Z}) in Theorem \ref{Th1}}:
Let $x\in \mathrm{H}^0(\frak{g}, Z_{\chi}(\lambda))$.
Since $\mathrm{H}^0(\frak{g}, Z_{\chi}(\lambda))$ is a weight-(super)module,
 $x$ may be viewed as one of the following forms:
 \begin{align*}
 &x_1(a+1,0,0)+x_2(a,1,1)\;\;\mbox{ or}\;\; x_3(a+1,0,1)+x_4(a,1,0)\;\;\mbox{if}\;\; \lambda_1\neq 0;\\
&y_1[a+1,0,0]+y_2[a,1,1]\;\;\mbox{or}\;\; y_3[a+1,0,1]+y_4[a,1,0]\;\;\mbox{if}\;\; \lambda=(0,\lambda_2)\neq 0;\\
&z_1(a+1,0,0) \;\;\mbox{or} \;\;z_2(a,1,0)\;\; \mbox{if}\;\;  \lambda= 0,
\end{align*}
where $x_i,y_i,z_i\in \mathbb{F}$.
By $Fx=0$, we have
$$x_1=x_3=y_1=y_3=z_1=0.$$

\textit{Case  $\lambda=(0,\lambda_2)\neq 0$}:
Note that
\begin{align*}
&0=Ey_2[a,1,1]\stackrel{(\ref{E002})}{=}y_2\lambda_2\left(a[a-1,1,0]+[a,0,1]\right),\\
&0=Ey_4[a,1,0]\stackrel{(\ref{E002})}{=}y_4\left(a[a-1,1,1]+\lambda_2[a,0,0]\right).
\end{align*}
It follows that $y_2=y_4=0$.

\textit{Case  $\lambda=0$}:
Note that
$$0=H_1z_2(a,1,0)\stackrel{(\ref{a10})}{=}z_2\left(\delta_{a=p-1}\chi(f)^p(0,0,0)+\delta_{a\neq p-1}(a+1,0,0)\right).$$
It follows that $z_2=0$ in case $a\neq p-1$ or $\chi(f)\neq 0$.
As a result, $x=z_2(p-1,1,0)$ if $\chi=0$.
It is clear that $(p-1,1,0)$ is in $\mathrm{H}^0(\frak{g}, Z_{\chi}(\lambda))$ in case $(\lambda,\chi)=(0,0)$.
Then $\mathrm{H}^0(\frak{g}, Z_{\chi}(\lambda))=\langle0\mid (p-1,1,0)\rangle$ if $(\lambda,\chi)=(0,0)$, and 0 otherwise.

\textit{Case  $\lambda_1=\lambda_2\neq0$  or
 $0\neq \lambda_1^2\neq \lambda_2^2\neq 0$}:
Note that
\begin{equation}
\begin{split}\label{2008231753}
&0=H_1x_2(a,1,1)\stackrel{(\ref{H1111})}{=}x_2\left(\delta_{a\neq p-1}(a+1,0,1)+\delta_{a=p-1}\chi(f)^P(0,0,1)-\lambda_1(a,1,0)\right),\\
&0=H_1x_4(a,1,0)\stackrel{(\ref{H1111})}{=}x_4\left(\delta_{a\neq p-1}(a+1,0,0)+\delta_{a=p-1}\chi(f)^P(0,0,0)-(a,1,1)\right).
\end{split}
\end{equation}
It follows that $x_2=x_4=0$.

\textit{Case  $\lambda=(\lambda_1,0)\neq 0$}:
By use of (\ref{H1010}), we may get the equation (\ref{2008231753}), which implies that  $x_2=x_4=0$.

\textit{Case  $\lambda=(\lambda_1,-\lambda_1)\neq 0$}:
By use of (\ref{a1k}), we may get the equation (\ref{2008231753}), which implies that  $x_2=x_4=0$.

It follows that
\begin{align}
\mathrm{H}^0(\frak{g}, {\mathrm{Z}}_{\chi}(\lambda))=\left\{\begin{array}{ll}
\langle 0 \mid (p-1,1,0)\rangle,  &(\chi,\lambda)=(0,0)\\
0,  &otherwise.
\end{array}\right.\label{0h}
\end{align}
As a result Formula (\ref{0Z}) is true.

\section{$\mathrm{H}^1(\frak{g},  Z_{\chi}(\lambda))$ and $\mathrm{H}^1(\frak{g}, L_{\chi}(\lambda))$ }

The following proposition determines the unique simple quotient $L_{\chi}(\lambda)$ of $Z_{\chi}(\lambda)$ in case $\lambda=(\lambda_1, -\lambda_1)\in \mathbb{F}_p^2$.
\begin{proposition}\label{200822839}
Let $\lambda=(\lambda_1, -\lambda_1)$.
\begin{itemize}
\item[(1)] Let $\lambda_1\in \mathbb{F}_p\backslash\left\{0\right\}$.
Denote by $M_1$ the subspace of $Z_{\chi}(\lambda)$ with a basis
$$
(a,1,1),\quad (c,0,0),\quad (c,0,1),\quad (c,1,0),\quad (b+1,0,1)-\lambda_1(b,1,0),\quad (2\lambda_1,1,0),
$$
where
$$0\leq a\leq p-1,\quad 0\leq b\leq \underline{2\lambda_1-1},\quad \underline{2\lambda_1+1}\leq c\leq p-1;$$
denote by  $M_2$ the subspace of $Z_{\chi}(\lambda)$ with a basis
$$
\left\{(a,1,1), (a+1,0,1)-\lambda_1(a,1,0)\mid 0\leq a\leq p-1\right\}.
$$
 Then
 $$L_{\chi}(\lambda)=\left\{\begin{array}{ll}
Z_{\chi}(\lambda)/M_1, &\chi=0\\
Z_{\chi}(\lambda)/M_2, &\chi\;\;\mbox{is  nilpotent},
\end{array}\right.$$
 \item[(2)] Let $\lambda=0$.
 Denote by $M_3$ the subspace of $Z_{\chi}(\lambda)$ with a basis
$$
\left\{(a,0,0),(a,1,0),(0,1,0)\mid 1\leq a\leq p-1\right\}.
$$
 Then
 $$L_{\chi}(\lambda)=\left\{\begin{array}{ll}
 Z_{\chi}(\lambda)/M_3, &\chi=0\\
  Z_{\chi}(\lambda), &\chi\;\;\mbox{is  nilpotent}.
\end{array}\right.$$
 \end{itemize}
\end{proposition}
\begin{proof}
Let $i=1,2$ or $3$.
From Remark \ref{200823907} (4),
it is routine to show that $M_i$ is a submodule of $Z_{\chi}(\lambda)$ under the corresponding condition.

It is sufficient to show that  $Z_{\chi}(\lambda)/M_i$
can be generated by any nonzero element.
 $M_i$ is a weight (super)module, so is $Z_{\chi}(\lambda)/M_i$.
 Then it is sufficient to show that $Z_{\chi}(\lambda)/M_i$
can be generated by any $\mathbb{Z}_2$-homogeneous weight vector.

Note that $Z_{0}(\lambda)/M_1$ has a basis consisting of the following $\mathbb{Z}_2$-homogeneous weight vectors
 $$\left\{(a,0,0), (a,0,1)\mid 0\leq a\leq \underline{2\lambda_1}\right\}.$$
Let $0\leq a\leq \underline{2\lambda_1}.$
By (\ref{e}), one sees that $Z_{0}(\lambda)/M_1$ can be generated  by $(a,0,0)$.
Then by $H_1(a,0,1)\stackrel{(\ref{a0k})}{=}\lambda_1(a,0,0)$ in $Z_{0}(\lambda)/M_1$ and $\lambda_1\neq 0$, one sees that
 $Z_{0}(\lambda)/M_1$ can also be generated by $(a,0,1)$.

Note that $Z_{\chi}(\lambda)/M_2$ has a basis consisting of the following $\mathbb{Z}_2$-homogeneous weight vectors
 $$\left\{(a,0,0), (a,0,1)\mid 0\leq a\leq p-1\right\}.$$
Let $0\leq a\leq p-1$.  Hence $Z_{\chi}(\lambda)/M_2$ can be generated  by $(a,0,0)$ in case $\chi$ is nilpotent
because of (\ref{f}).
Then by $H_1(a,0,1)\stackrel{(\ref{a0k})}{=}\lambda_1(a,0,0)$ in $Z_{\chi}(\lambda)/M_2$ and $\lambda_1\neq 0$, one sees that
 $Z_{\chi}(\lambda)/M_2$ can also be generated by $(a,0,1)$.

Note that $Z_{\chi}(\lambda)/M_3$ has a basis $\{(0,0,0)\}$,
which can generate $Z_{\chi}(\lambda)/M_3$ obviously.

Let $\chi$ be nilpotent.
Note that $Z_{\chi}(0)$ has a basis consisting of the following $\mathbb{Z}_2$-homogeneous weight vectors
$$\{(a,0,0),(a,1,0)\mid 0\leq a\leq p-1\}.$$
Let $0\leq a\leq p-1$. Hence $Z_{\chi}(0)$ can be generated  by $(a,0,0)$ in case $\chi$ is nilpotent
because of (\ref{f}).
Then by (\ref{a10}), $Z_{\chi}(0)$ can be generated by $(a,1,0)$.
\end{proof}

\subsection{Target-weight spaces}
In view of Lemma \ref{reduction}, in order to determine $\mathrm{Der}(\frak{g}, Z_{\chi}(\lambda))_{(0)}$ and $\mathrm{Der}(\frak{g},L_{\chi}(\lambda))_{(0)}$, we shall give the weight-spaces $ Z_{\chi}(\lambda)_{\alpha}$
and  $ Z_{\chi}(\lambda)_{\alpha}$ for $\alpha=(1,-1),(-1,1),0$.
The weights $(1,-1),(-1,1)$ and $0$
are called the target-weights of $ Z_{\chi}(\lambda)$ or $L_{\chi}(\lambda)$.
The following two lemmas determine all the target-weight spaces of
$ Z_{\chi}(\lambda)$ and $L_{\chi}(\lambda)$.
\begin{lemma}\label{weightspace}
For $\lambda=(\lambda_1,\lambda_2)$, the target-weight spaces of
$ Z_{\chi}(\lambda)$ and $L_{\chi}(\lambda)$ are listed below.
\begin{itemize}
\item[(1)]
Let $\lambda_1=-\lambda_2\notin \mathbb{F}_p$ or $\lambda_1\neq-\lambda_2$. Then
$$ Z_{\chi}(\lambda)_{\alpha}=L_{\chi}(\lambda)_{\alpha}=0, \quad \alpha=0,\; (1,-1),\;(-1,1).$$

\item[(2)]
Let $\lambda_1=-\lambda_2\in \mathbb{F}_p\backslash\{0\}$. Then
\begin{align*}
 Z_{\chi}(\lambda)_{\alpha}&=\left\{\begin{array}{ll}\langle (\lambda_1,0,0),(\lambda_1-1,1,1)\mid  (\lambda_1-1,1,0),(\lambda_1,0,1)\rangle, &\alpha=0\\
\langle (\lambda_1+1,0,0),(\lambda_1,1,1) \mid (\lambda_1,1,0),(\lambda_1+1,0,1)\rangle,  &\alpha=(-1,1)\\
\langle (\lambda_1-1,0,0),(\lambda_1-2,1,1)\mid (\lambda_1-2,1,0),(\lambda_1-1,0,1)\rangle,  &\alpha=(1,-1).
\end{array}\right.
\end{align*}

\item[(3)]
Let $\lambda=0$. Then
\begin{align*}
 Z_{\chi}(\lambda)_{\alpha}&=\left\{\begin{array}{ll}
\langle (0,0,0)\mid (p-1,1,0)\rangle, &\alpha=0\\
\langle (p-1,0,0)\mid (p-2,1,0)\rangle, &\alpha=(1,-1)\\
\langle (1,0,0)\mid (0,1,0)\rangle, &\alpha=(-1,1).
\end{array}\right. \\
L_{0}(\lambda)_{\alpha}&=\left\{\begin{array}{ll}
\langle (0,0,0)\mid 0\rangle, &\alpha=0\\
0, &\alpha=(1,-1), (-1,1).
\end{array}\right.
\end{align*}
\end{itemize}
\end{lemma}
\begin{proof}
Note that in $Z_{\chi}(\lambda)$,
$(a,j,k)$ or
$[a,j,k]$ has the weight
$(\lambda_1-a-j, \lambda_2+a+j)$, since
$$h_i(a,j,k)=\left(\delta_{i=1}(\lambda_1-a-j)+\delta_{i=2}(\lambda_2+a+j)\right)(a,j,k).$$
Then the condition $\lambda_1+\lambda_2\neq 0$ or $\lambda_1=-\lambda_2\notin \mathbb{F}_p$ implies that
$$ Z_{\chi}(\lambda)_{\alpha}=0, \quad \alpha=0,\; (1,-1),\;(-1,1).$$
Then
$$L_{\chi}(\lambda)_{\alpha}=0, \quad \alpha=0,\; (1,-1),\;(-1,1)$$
in case $\lambda_1+\lambda_2\neq 0$ or $\lambda_1=-\lambda_2\notin \mathbb{F}_p$.
Hence (1) is true.
By a direct computation, the conclusions on $ Z_{\chi}(\lambda)$ in (2) and (3) are true from Remark \ref{basis} and (\ref{h}).

Let $\lambda=0$.
Then the conclusions on $L_{\chi}(\lambda)$ in (3) are true from  Proposition \ref{200822839}.
\end{proof}

\begin{lemma}\label{weightspacel}
Let $\lambda=(\lambda_1,-\lambda_1)\in \mathbb{F}_p^2$.
\begin{itemize}
\item[(1)]
If $\lambda_1=1$, then
$$
L_{0}(\lambda)_{\alpha}=\left\{\begin{array}{ll}\langle (1,0,0)\mid (0,1,0)\rangle, &\alpha=0\\
\langle (2,0,0) \mid (1,1,0)\rangle,  &\alpha=(-1,1)\\
\langle (0,0,0)\mid (0,0,1)\rangle,  &\alpha=(1,-1).
\end{array}\right.
$$

\item[(2)]
If $\lambda_1=p-1$, then
$$
L_{0}(\lambda)_{\alpha}=\left\{\begin{array}{ll}0, &\alpha=0\\
\langle (0,0,0) \mid (0,0,1)\rangle,  &\alpha=(-1,1)\\
\langle (p-2,0,0)\mid (p-3,1,0)\rangle,  &\alpha=(1,-1).
\end{array}\right.
$$

\item[(3)]
If $p\geq 5$ and  $2\leq\lambda_1\leq \frac{p-1}{2}$, then
$$
L_{0}(\lambda)_{\alpha}=\left\{\begin{array}{ll}\langle (\lambda_1,0,0)\mid (\lambda_1-1,1,0)\rangle, &\alpha=0\\
\langle (\lambda_1+1,0,0) \mid (\lambda_1,1,0)\rangle,  &\alpha=(-1,1)\\
\langle (\lambda_1-1,0,0)\mid (\lambda_1-2,1,0)\rangle,  &\alpha=(1,-1).
\end{array}\right.
$$

\item[(4)]
If $p\geq 5$ and   $\frac{p+1}{2}\leq \lambda_1\leq p-2$, then
$L_{0}(\lambda)_{\alpha}=0$ for $\alpha=0,(-1,1),(1,-1).$

\item[(5)] Let $\chi$ be nilpotent and $\lambda_1\in\mathbb{F}_p\backslash\{0\}$.
Then
$$
L_{\chi}(\lambda)_{\alpha}=\left\{\begin{array}{ll}\langle (\lambda_1,0,0)\mid (\lambda_1-1,1,0)\rangle, &\alpha=0\\
\langle (\lambda_1+1,0,0) \mid (\lambda_1,1,0)\rangle,  &\alpha=(-1,1)\\
\langle (\lambda_1-1,0,0)\mid (\lambda_1-2,1,0)\rangle,  &\alpha=(1,-1).
\end{array}\right.
$$

\item[(6)] Let $(\chi,\lambda)=(0,0)$. Then
$$
L_{\chi}(\lambda)_{\alpha}=\left\{\begin{array}{ll}\langle (0,0,0)\mid 0\rangle, &\alpha=0\\
0,  &\alpha=(-1,1),(1,-1).
\end{array}\right.
$$
\end{itemize}
\end{lemma}
\begin{proof}
\textit{Case $\chi=0$}:
From Proposition \ref{200822839} (5), $L_{0}(\lambda)$
has a basis
$$\left\{(a,0,0),(2\lambda_1,0,0), (a,1,0),(0,0,1)\mid 0\leq a \leq\underline{ 2\lambda_1-1}\right\}$$
and the following equations hold in $L_{0}(\lambda)$:
\begin{align*}
(a+1,0,1)&=\lambda_1(a,1,0)\neq 0, \quad 0\leq a \leq \underline{2\lambda_1-1},\\
(b+1,0,1)&=\lambda_1(b,1,0)=0, \quad \underline{2\lambda_1} \leq b \leq p-2,\\
(p-1,1,0)&=(c,0,0)=(d,1,1)=0, \quad \underline{2\lambda_1+1} \leq c \leq p-1, 0\leq d\leq p-1.
\end{align*}
Then for $x=0, \pm 1, -2$, it is necessary to compare $\underline{\lambda_1+x}$ and $\underline{2\lambda_1}$ by Lemma \ref{weightspace} (2).
To that aim, we get the following conclusions:
\begin{itemize}
\item If $p\geq 5$ and  $2\leq\lambda_1\leq \frac{p-1}{2}$, then
$$\underline{\lambda_1-2}<\underline{\lambda_1-1}<\underline{\lambda_1}<\underline{\lambda_1+1}<\underline{2\lambda_1}.$$
\item If $p\geq 5$ and  $\lambda_1=1$, then
$$\underline{\lambda_1-1}<\underline{\lambda_1}<\underline{\lambda_1+1}=\underline{2\lambda_1}<\underline{\lambda_1-2}.$$
\item If $p\geq 5$ and  $\lambda_1=p-1$, then
$$\underline{\lambda_1+1}<\underline{\lambda_1-2}<\underline{\lambda_1-1}=\underline{2\lambda_1}<\underline{\lambda_1}.$$
\item If $p\geq 5$ and  $\lambda_1=p-2$, then
$$\underline{\lambda_1-2}=\underline{2\lambda_1}<\underline{\lambda_1-1}<\underline{\lambda_1}<\underline{\lambda_1+1}.$$
\item If $p\geq 7$ and  $\frac{p+1}{2}\leq\lambda_1\leq p-3$, then
$$\underline{2\lambda_1}<\underline{\lambda_1-2}<\underline{\lambda_1-1}<\underline{\lambda_1}<\underline{\lambda_1+1}.$$
\item If $p=3$ and  $\lambda_1=1$, then
$$\underline{\lambda_1-1}<\underline{\lambda_1}<\underline{\lambda_1+1}=\underline{2\lambda_1}=\underline{\lambda_1-2}.$$
\item If $p=3$ and  $\lambda_1=2$, then
$$\underline{\lambda_1+1}=\underline{\lambda_1-2}<\underline{\lambda_1-1}=\underline{2\lambda_1}<\underline{\lambda_1}.$$
\end{itemize}
Therefore, the conclusions on target-weight spaces of $L_{0}(\lambda)$ in (1)--(4) are true.

\textit{Case $\chi$ being nilpotent}:
From Proposition \ref{200822839} (5), $L_{\chi}(\lambda)$
has a basis
$$\left\{(a,0,0), (a,1,0)\mid 0\leq a \leq p-1\right\}$$
and the following equations hold in $L_{\chi}(\lambda)$
$$
(a+1,0,1)=\lambda_1(a,1,0)\neq 0, \quad(a,1,1)=0, \quad 0\leq a \leq p-1.
$$
Then (5) holds.
\end{proof}

\subsection{Weight-derivation spaces}
In this subsection, we determine all of weight-derivations from $\frak{g}$ to $ Z_{\chi}(\lambda)$ or $L_{\chi}(\lambda)$.

From Remark \ref{200823907} and Proposition \ref{200822839}, we get the following formulas
about $L_{\chi}(\lambda)$ as the $\frak{g}$-module in case
 $\lambda=(\lambda_1,-\lambda_1)\in \mathbb{F}_p^2\backslash\{0\}$,
 which will be used in the future and only need a direct computation:
\begin{align}
&H_1(0,0,1)=H_2(0,0,1)=\lambda_1(0,0,0)\label{001},\\
&F(a,j,0)=\delta_{j=0}\left(\delta_{\chi=0}\delta_{0\leq a\leq \underline{2\lambda_1-1}}+\delta_{\chi\neq 0}\right)(a,1,0)\label{Fl},
\end{align}
\begin{equation}
\begin{split}\label{a10l}
&H_i(a,1,0)=\left(\delta_{\chi=0}\delta_{0\leq a\leq \underline{2\lambda_1-1}}+\delta_{\chi\neq 0}\right)\delta_{a\neq p-1}(a+1,0,0)\\
&\;\;\quad\quad\quad\quad\quad+\delta_{\chi\neq 0}\delta_{a= p-1}\chi(f)^p(0,0,0),
\end{split}
\end{equation}
\begin{equation}
\begin{split}\label{a00l}
&H_i(a,0,0)=\left(\delta_{\chi=0}\delta_{1\leq a\leq \underline{2\lambda_1}}+\delta_{\chi\neq 0}\right)(-1)^i(a-\lambda_1)(a-1,1,0)\\
&\;\;\quad\quad\quad\quad\quad+\delta_{(a,\chi)=(0, 0)}(-1)^{i+1}(0,0,1),
\end{split}
\end{equation}
\begin{equation}
\begin{split}\label{El}
&E(a,j,0)=\delta_{j=0}\left(\delta_{\chi=0}\delta_{2\leq a\leq \underline{2\lambda_1}}+\delta_{\chi\neq 0}\right)a(2\lambda_1-a+1)(a-2,1,0)\\
&\quad\quad\quad\quad\quad+\delta_{(j,\chi,a)=(0,0,1)}2(0,0,1),
\end{split}
\end{equation}
\begin{equation}
\begin{split}\label{fl}
&f(a,j,0)=\delta_{a\neq p-1}\left(\delta_{\chi=0}\delta_{0\leq a< \underline{2\lambda_1-1}}+\delta_{j=0}\delta_{\chi=0}\delta_{a= \underline{2\lambda_1-1}}\right)(a+1,j,0)\\
&\quad\quad\quad\quad\quad+\delta_{\chi\neq 0}\delta_{a\neq p-1}(a+1,j,0)+\delta_{a= p-1}\delta_{\chi\neq 0}\chi(f)^p(0,0,0),
\end{split}
\end{equation}
\begin{equation}
\begin{split}\label{el}
&e(a,j,0)=a\left(2\lambda_1-a+(-1)^j\right)\left(\delta_{(\chi,k,a)=(0,0,\underline{2\lambda_1})}+\delta_{\chi\neq 0}\right)(a-1,j,0)\\
&\quad\quad\quad\quad\quad+a\left(2\lambda_1-a+(-1)^j\right)\delta_{\chi=0}\delta_{0\leq a\leq \underline{2\lambda_1-1}}(a-1,j,0)\\
&\quad\quad\quad\quad\quad+\delta_{(\chi,j,k)=(0,1,0)}\left(2\delta_{a\neq 0}\lambda_1(a-1,1,0)+2\delta_{a= 0}(0,0,1)\right)\\
&\quad\quad\quad\quad\quad+2\delta_{\chi\neq 0}\lambda_1(a-1,1,0).
\end{split}
\end{equation}

\begin{lemma}\label{1649}
\begin{itemize}
\item[(1)]
If $\chi$ is semisimple, then
$$\mathrm{Der}\left(\frak{g}, Z_{\chi}(\lambda)\right)_{(0)}
=0.$$

\item[(2)]
Let $\sigma\in \mathrm{Der}\left(\frak{g}, Z_{\chi}(\lambda)\right)_{(0)}$ and
$\tau\in \mathrm{Der}\left(\frak{g},L_{\chi}(\lambda)\right)_{(0)}$.
Then
$$\sigma(h_i)=\delta_{(\chi,\lambda)\neq (0,0)}\tau(h_i)=0, \quad i=1,2.$$

\item[(3)]
Let $\sigma\in \mathrm{Der}\left(\frak{g}, Z_{\chi}(\lambda)\right)_{(0)}$ and
$\tau\in \mathrm{Der}\left(\frak{g},L_{\chi}(\lambda)\right)_{(0),\bar{1}}$.
If
$$\mbox{$(i_1,j_1,k_1)\neq (\lambda_1,1,0), (i_2,j_2,k_2)\neq (\lambda_1,1,1)$ and $(i_3,j_3,k_3)\neq (\lambda_1-2,1,1)$,}$$
 then
\begin{align*}
&\delta_{(\chi,\lambda_1)\neq(0,1)}\tau(F)=\delta_{(\chi,\lambda_1)\neq(0,1)}\tau(E)=0,\\
&\delta_{(\chi,\lambda_1)=(0,1)}\left(2\tau(F)^{(\lambda_1+1,0,0)}+\tau(E)^{(\lambda_1-1,0,0)}\right)=0,\\
&\delta_{|\sigma|=\bar{0}}\sigma(F)^{(i_1,j_1,k_1)}=\delta_{|\sigma|=\bar{1}}\sigma(F)^{(i_2,j_2,k_2)}
=\delta_{|\sigma|=\bar{1}}\sigma(E)^{(i_3,j_3,k_3)}=0.
\end{align*}

\end{itemize}
\end{lemma}
\begin{proof}
(1) Let $\chi$ be semisimple.
There exists $h_i$ such that $\chi(h_i)\neq 0$.
Since $\lambda\in \Lambda_{\chi}$, that is, $\lambda_i^p-\lambda_i=\chi(h_i)^p$,
we get $\lambda_i\notin \mathbb{F}_p$.
From Lemma \ref{weightspace} (1),
$$\mathrm{Hom}_{\mathbb{F}}\left(\frak{g}, Z_{\chi}(\lambda)\right)_{(0)}
=0.$$
It follows that
$$\mathrm{Der}\left(\frak{g}, Z_{\chi}(\lambda)\right)_{(0)}
=0.$$

(2)
From Lemma \ref{reduction}, (\ref{0h}) and Lemma \ref{weightspace}, it is sufficient to show
$$\delta_{|\sigma|=\bar{1}}\sigma(h_i)^{(p-1,1,0)}=0,\quad  i=1,\;2$$
in case $(\chi, \lambda)=(0,0)$.
Let $(\lambda,\chi)=(0,0)$ and $|\sigma|=\bar{1}$.
By the definition of derivations,
$$2\sigma(h_i)=\sigma([H_i,H_i])=-2H_i\sigma(H_i).$$
Then
$$-\sigma(h_i)^{(p-1,1,0)}(p-1,1,0)=\sigma(H_i)^{(0,0,0)}H_{i}(0,0,0)\stackrel{(\ref{a00})}{=}0.$$
As a result,  $\sigma(h_i)^{(p-1,1,0)}=0$ in case $(\lambda,\chi)=(0,0)$ and $|\sigma|=\bar{1}$.

(3)
From Lemmas \ref{weightspace} and \ref{weightspacel}, it is sufficient to show
\begin{align*}
&\delta_{|\sigma|=\bar{0}}\delta_{\lambda_1\neq 0}\sigma(F)^{(\lambda_1+1,0,1)}=\delta_{|\sigma|=\bar{1}}\sigma(F)^{(\lambda_1+1,0,0)}=
\delta_{|\sigma|=\bar{1}}\sigma(E)^{(\lambda_1-1,0,0)}=0,\\
&\delta_{(\chi,\lambda_1)\neq(0,1)}\tau(F)^{(\lambda_1+1,0,0)}
=\delta_{(\chi,\lambda_1)\neq(0,1)}\tau(E)^{(\lambda_1-1,0,0)}=0,\\
&\delta_{(\chi,\lambda_1)=(0,1)}\left(2\tau(F)^{(\lambda_1+1,0,0)}+\tau(E)^{(\lambda_1-1,0,0)}\right)=0,
\end{align*}
in case $\lambda_1=-\lambda_2\in \mathbb{F}_p$.
Let $\lambda_1=-\lambda_2\in \mathbb{F}_p$ and $|\tau|=\bar{1}$.
By the definition of weight-derivations and $[F,F]=0$,
\begin{align*}
0=F\sigma(F)&=\delta_{|\sigma|=\bar{0}}\left(\delta_{\lambda_1\neq0}\sigma(F)^{(\lambda_1+1,0,1)}F(\lambda_1+1,0,1)
+\sigma(F)^{(\lambda_1,1,0)}F(\lambda_1,1,0)\right)\\
&\quad\;+\delta_{|\sigma|=\bar{1}}\left(\delta_{\lambda_1\neq0}\sigma(F)^{(\lambda_1,1,1)}F(\lambda_1,1,1)
+\sigma(F)^{(\lambda_1+1,0,0)}F(\lambda_1+1,0,0)\right)\\
&\stackrel{(\ref{F})}{=}\delta_{|\sigma|=\bar{0}}\delta_{\lambda_1\neq0}\sigma(F)^{(\lambda_1+1,0,1)}(\lambda_1+1,1,1)
+\delta_{|\sigma|=\bar{1}}\sigma(F)^{(\lambda_1+1,0,0)}(\lambda_1+1,1,0),
\end{align*}
$$0=F\tau(F)\stackrel{(\ref{Fl})}{=}\tau(F)^{(\lambda_1+1,0,0)}
(\lambda_1+1,1,0).$$
Hence
$$
\left(\delta_{\chi=0}\delta_{\lambda_1\neq \pm 1}+\delta_{P_1}\right)\tau(F)^{(\lambda_1+1,0,0)}=0,$$
$$\delta_{|\sigma|=\bar{0}}\delta_{\lambda_1\neq 0}\sigma(F)^{(\lambda_1+1,0,1)}=\delta_{|\sigma|=\bar{1}}\sigma(F)^{(\lambda_1+1,0,0)}=0,$$
where the first equation is from
\begin{align*}
(\lambda_1+1,1,0)\neq 0\;\mbox{in}\; L_{\chi}(\lambda)\Longleftrightarrow &
\chi=0\;\mbox{and}\;\lambda_1\notin \{1,p-1\},\\
&\mbox{or }\;\chi \;\mbox{is nilpotent}.
\end{align*}

For convenience, hereafter denote by
$P_1$ the position that $\chi$ is nilpotent;
$P_4$ the one that $2\leq\lambda_1\leq \frac{p-1}{2}$.

Let $|\sigma|=|\tau|=\bar{1}$ in the following.
Then by $[E,F]=h_1+h_2$ and (2),
\begin{align*}
0=-\sigma([E,F])=&\delta_{\lambda_1\neq0}\sigma(F)^{(\lambda_1,1,1)}E(\lambda_1,1,1)\\
&+\delta_{\lambda_1\neq0}\sigma(E)^{(\lambda_1-2,1,1)}F(\lambda_1-2,1,1)\\
&+\sigma(E)^{(\lambda_1-1,0,0)}F(\lambda_1-1,0,0)\\
\stackrel{(\ref{F}),(\ref{E})}{=}&\sigma(E)^{(\lambda_1-1,0,0)}F(\lambda_1-1,1,0),\\
0=-\tau([E,F])=&
\left(\delta_{(\chi,\lambda_1)=(0,1)}+\delta_{(\chi,\lambda_1)=(0,p-1)}\right)2\tau(F)^{(\lambda_1+1,0,0)}E(\lambda_1+1,0,0)\\
&+\tau(E)^{(\lambda_1-1,0,0)}F(\lambda_1-1,0,0)\\
\stackrel{(\ref{Fl}),(\ref{El})}{=}&\delta_{(\chi,\lambda_1)=(0,1)}\left(2\tau(F)^{(\lambda_1+1,0,0)}+\tau(E)^{(\lambda_1-1,0,0)}\right)(\lambda_1-1,1,0)\\
&+\delta_{\chi=0}\delta_{P_4}\tau(E)^{(\lambda_1-1,0,0)}(\lambda_1-1,1,0)\\
&+\delta_{P_1}\tau(E)^{(\lambda_1-1,0,0)}(\lambda_1-1,1,0),
\end{align*}
where the last equation is from
\begin{align*}
(\lambda_1-1,1,0)\neq 0\;\mbox{in}\; L_{\chi}(\lambda)\Longleftrightarrow &
\chi=0\;\mbox{and}\;1\leq \lambda_1\leq \frac{p-1}{2},\\
&\mbox{or }\;\chi \;\mbox{is nilpotent}.
\end{align*}
It follows that
\begin{align*}
&\delta_{|\sigma|=\bar{1}}\sigma(E)^{(\lambda_1-1,0,0)}=0,\\
&\delta_{(\chi,\lambda_1)=(0,1)}\left(2\tau(F)^{(\lambda_1+1,0,0)}+\tau(E)^{(\lambda_1-1,0,0)}\right)=0,\\
&\delta_{\chi=0}\delta_{P_4}\tau(E)^{(\lambda_1-1,0,0)}=0,\\
&\delta_{P_1}\tau(E)^{(\lambda_1-1,0,0)}=0.
\end{align*}
It remains to show
$\delta_{(\chi,\lambda_1)=(0,p-1)}\tau(E)^{(\lambda_1-1,0,0)}=0$.
By $[E,E]=0$ and the definition of derivations,
\begin{align*}
0=E\tau(E)=&\tau(E)^{(\lambda_1-1,0,0)}E(\lambda_1-1,0,0)\left(\delta_{(\chi,\lambda_1)=(0, p-1)}+\delta_{(\chi,\lambda_1)=(0, 1)}\right)\\
\stackrel{(\ref{El})}{=}&\delta_{(\chi,\lambda_1)=(0, p-1)}\tau(E)^{(\lambda_1-1,0,0)}\left(\delta_{p=3}2(0,0,1)-\delta_{p\geq 5}2(\lambda_1-3,1,0)\right).
\end{align*}
It follows that
$$\delta_{(\chi,\lambda_1)=(0, p-1)}\tau(E)^{(\lambda_1-1,0,0)}=0.$$
\end{proof}

For convenience, define two linear maps from $\frak{g}$ to $Z_{\chi}(\lambda)$ as follows:
\begin{itemize}
\item[(1)] Define $\varphi\in \mathrm{Hom}_{\mathbb{F}}(\frak{g}, Z_{\chi}(\lambda))$ by
$$\varphi(H_1)=\varphi(H_2)=(p-1,1,0),\quad \varphi(x)=0, $$
where $x=e, E, F, f, h_1, h_2.$

\item[(2)] Define $\psi\in \mathrm{Hom}_{\mathbb{F}}(\frak{g}, Z_{\chi}(\lambda))$ by
$$\psi(H_1)=\psi(H_2)=-(0,0,0),\quad \psi(f)=(0,1,0),\quad \psi(x)=0,$$
where $x=e, E, F, h_1, h_2.$
\end{itemize}

In the following, we give a proof of the second part of Theorem \ref{Th1}.

\textit{The proof of Formula (\ref{1Z}) in Theorem \ref{Th1}}:
Claim that
$$
\mathrm{Der}(\frak{g}, Z_{\chi}(\lambda))_{(0)}=\left\{\begin{array}{ll}
\langle \frak{D}_{(\lambda_1-1,1,1)}, \frak{D}_{(\lambda_1,0,0)}\mid\frak{D}_{(\lambda_1-1,1,0)}, \frak{D}_{(\lambda_1,0,1)}\rangle, &\lambda=(\lambda_1,-\lambda_1)\in\mathbb{F}_p^2\backslash\{0\}\\
\langle \frak{D}_{(0,0,0)}, \varphi\mid \psi\rangle, &(\lambda,\chi)=(0,0)\\
\langle \frak{D}_{(0,0,0)}\mid \frak{D}_{(p-1,1,0)}\rangle, &\lambda=0, \chi\;\mbox{is nilpotent}\\
0, & \mbox{otherwise}.
\end{array}\right.
$$

It is routine to show that $\varphi, \phi$ are weight-derivations if $(\lambda,\chi)=(0,0)$.
It is a standard fact that the space
$\mathrm{Ider}(\frak{g}, Z_{\chi}(\lambda))_{(0)}=\langle\frak{D}_{(0,0,0)}\mid 0\rangle$ in case $(\lambda,\chi)=(0,0)$.
which implies that the derivations $\varphi, \psi$ are not inner.

Firstly we show that the elements are linearly independent in the right sets of the above claim.

\textit{Case $\lambda=(\lambda_1,-\lambda_1)\in \mathbb{F}_p^2\backslash\{0\}$}:
On one hand, let
$$a\frak{D}_{(\lambda_1-1,1,1)}+b\frak{D}_{(\lambda_1,0,0)}=0=c\frak{D}_{(\lambda_1,0,1)}+d\frak{D}_{(\lambda_1-1,1,0)},$$
where $a,b,c,d\in \mathbb{F}$. For convenience, we write
$$\sigma=a\frak{D}_{(\lambda_1-1,1,1)}+b\frak{D}_{(\lambda_1,0,0)},\quad \tau=c\frak{D}_{(\lambda_1,0,1)}+d\frak{D}_{(\lambda_1-1,1,0)}.$$
Then
\begin{align*}
&0=\sigma(F)=aF(\lambda_1-1,1,1)+bF(\lambda_1,0,0)\stackrel{(\ref{F})}{=}b(\lambda_1,1,0),\\
&0=\tau(F)=cF(\lambda_1,0,1)+dF(\lambda_1-1,1,0)\stackrel{(\ref{F})}{=}c(\lambda_1,1,1).
\end{align*}
which implies $b=0=c$.
Furthermore,
\begin{align*}
&0=\sigma(f)=af(\lambda_1-1,1,1)\stackrel{(\ref{f})}{=}a(\lambda_1,1,1),\\
&0=\tau(f)=df(\lambda_1-1,1,0)\stackrel{(\ref{F})}{=}d(\lambda_1,1,0).
\end{align*}
which implies $a=0=d$.
Then
$$\mbox{$\left\{\frak{D}_{(\lambda_1-1,1,1)}, \frak{D}_{(\lambda_1,0,0)}\right\}$ or $\left\{\frak{D}_{(\lambda_1,0,1)}, \frak{D}_{(\lambda_1-1,1,0)}\right\}$}$$
is linearly independent in case $\lambda=(\lambda_1,-\lambda_1)\in \mathbb{F}_p^2\backslash\{0\}$.

\textit{Case $\lambda=0$}:
On one hand, it is routine to show that $\varphi$ and $\psi$ are weight-derivations.
It is obvious that $\varphi,\psi$ and $\frak{D}_{(0,0,0)}$ are linearly independent in case $(\lambda,\chi)=(0,0)$,
so are $\psi=\frak{D}_{(p-1,1,0)}$ and $\varphi=\frak{D}_{(0,0,0)}$ in case $\lambda=0$ and $\chi$ is nilpotent.

Secondly we prove that weight-derivations must be in the right sets of the above claim
in case  $\lambda=(\lambda_1, -\lambda_1)\in \mathbb{F}_p^2$.

Let $\sigma\in \mathrm{Der}(\frak{g},  Z_{\chi}(\lambda))_{(0),\bar{0}}$.
On one hand, by the definition of derivations, Lemmas \ref{weightspace} and \ref{1649} (3), we get the following equations:
\begin{align*}
\sigma([H_i, F])=&\sigma(F)^{(\lambda_1,1,0)}H_i(\lambda_1,1,0)+\delta_{\lambda_1\neq 0}\sigma(H_i)^{(\lambda_1,0,1)}F(\lambda_1,0,1)\\
&+\sigma(H_i)^{(\lambda_1-1,1,0)}F(\lambda_1-1,1,0)\\
\stackrel{(\ref{F}),(\ref{a1k}),(\ref{a10})}{=}&\delta_{\lambda_1\neq 0}\left((-1)^i\sigma(F)^{(\lambda_1,1,0)}+\sigma(H_i)^{(\lambda_1,0,1)}\right)(\lambda_1,1,1)\\
&+\delta_{\lambda_1\neq p-1}\sigma(F)^{(\lambda_1,1,0)}(\lambda_1+1,0,0)+\delta_{\lambda_1=p-1}\sigma(F)^{(\lambda_1,1,0)}\chi(f)^p(0,0,0),
\end{align*}
\begin{align*}
\sigma([e,H_i])=&\delta_{\lambda_1\neq 0}\left(\sigma(H_i)^{(\lambda_1,0,1)}e(\lambda_1,0,1)-\sigma(e)^{(\lambda_1-2,1,1)}H_i(\lambda_1-2,1,1)\right)\\
&+\sigma(H_i)^{(\lambda_1-1,1,0)}e(\lambda_1-1,1,0)-\sigma(e)^{(\lambda_1-1,0,0)}H_i(\lambda_1-1,0,0)\\
\stackrel{(\ref{e})-(\ref{a0k}),(\ref{e0}),(\ref{a10})}{=}&\delta_{\lambda_1\neq 0}\left(\sigma(H_i)^{(\lambda_1,0,1)}\lambda_1(\lambda_1+1)+2\sigma(H_i)^{(\lambda_1-1,1,0)}\right)(\lambda_1-1,0,1)   \\
&+\delta_{\lambda_1\neq 0}\left(-\delta_{\lambda_1\neq 1}\sigma(e)^{(\lambda_1-2,1,1)}+(-1)^i\sigma(e)^{(\lambda_1-1,0,0)}\right)(\lambda_1-1,0,1)\\
&+\delta_{\lambda_1\neq 0}\left(\sigma(H_i)^{(\lambda_1-1,1,0)}\lambda_1(\lambda_1-1)+\sigma(e)^{(\lambda_1-2,1,1)}\lambda_1\right)(\lambda_1-2,1,0)\\
&+\delta_{\lambda_1\neq 0}\left(-(-1)^i\sigma(e)^{(\lambda_1-1,0,0)}(\lambda_1-1)\right)(\lambda_1-2,1,0)\\
&-\delta_{\lambda_1=1}\sigma(e)^{(\lambda_1-2,1,1)}\chi(f)^p(0,0,1)+\delta_{\lambda_1=0}(-1)^i\sigma(e)^{(\lambda_1-1,0,0)}(p-2,1,0),
\end{align*}
\begin{align*}
\sigma([e,f])=&\delta_{\lambda_1\neq 0}\left(\sigma(f)^{(\lambda_1,1,1)}e(\lambda_1,1,1)-\sigma(e)^{(\lambda_1-2,1,1)}f(\lambda_1-2,1,1)\right)\\
&+\sigma(f)^{(\lambda_1+1,0,0)}e(\lambda_1+1,0,0)-\sigma(e)^{(\lambda_1-1,0,0)}f(\lambda_1-1,0,0)\\
\stackrel{(\ref{f}),(\ref{e})}{=}&\delta_{\lambda_1\neq 0}\left(\sigma(f)^{(\lambda_1,1,1)}\lambda_1(\lambda_1-1)-\delta_{\lambda_1\neq 1}\sigma(e)^{(\lambda_1-2,1,1)}\right)(\lambda_1-1,1,1)\\
&+\delta_{\lambda_1\neq 0}\left(\sigma(f)^{(\lambda_1+1,0,0)}\lambda_1(\lambda_1+1)-\sigma(e)^{(\lambda_1-1,0,0)}\right)(\lambda_1,0,0)\\
&-\delta_{\lambda_1=1}\sigma(e)^{(\lambda_1-2,1,1)}\chi(f)^p(0,1,1)-2\delta_{\lambda_1=0}\sigma(e)^{(\lambda_1-1,0,0)}\chi(f)^p(0,0,0),
\end{align*}
\begin{align*}
\sigma([f,E])=&\delta_{\lambda_1\neq 0}\left(\sigma(f)^{(\lambda_1,1,1)}E(\lambda_1,1,1)+\sigma(E)^{(\lambda_1-1,0,1)}f(\lambda_1-1,0,1)\right)\\
&-\sigma(f)^{(\lambda_1+1,0,0)}E(\lambda_1+1,0,0)+\sigma(E)^{(\lambda_1-2,1,0)}f(\lambda_1-2,1,0)\\
\stackrel{(\ref{f}),(\ref{E})}{=}&\delta_{\lambda_1\neq 0}\left(\sigma(E)^{(\lambda_1-1,0,1)}-2\sigma(f)^{(\lambda_1+1,0,0)}(\lambda_1+1)\right)(\lambda_1,0,1)   \\
&+\delta_{\lambda_1\neq 0}\left(\delta_{\lambda_1\neq 1}\sigma(E)^{(\lambda_1-2,1,0)}+\sigma(f)^{(\lambda_1+1,0,0)}\lambda_1(\lambda_1+1)\right)(\lambda_1-1,1,0)\\
&+\delta_{\lambda_1=1}\sigma(E)^{(\lambda_1-2,1,0)}\chi(f)^p(0,1,0)+\delta_{\lambda_1=0}\sigma(E)^{(\lambda_1-2,1,0)}(p-1,1,0),
\end{align*}
\begin{align*}
\sigma([H_i,H_i])=&2\delta_{\lambda_1\neq 0}\sigma(H_i)^{(\lambda_1,0,1)}H_i(\lambda_1,0,1)+2\sigma(H_i)^{(\lambda_1-1,1,0)}H_i(\lambda_1-1,1,0)\\
\stackrel{(\ref{a1k}),(\ref{a0k}),(\ref{a10})}{=}&2\delta_{\lambda_1\neq 0}\left(\sigma(H_i)^{(\lambda_1,0,1)}\lambda_1+\sigma(H_i)^{(\lambda_1-1,1,0)}\right)\left((-1)^i(\lambda_1-1,1,1)+ (\lambda_1,0,0)\right)  \\
&+2\delta_{\lambda_1=0}\sigma(H_i)^{(\lambda_1-1,1,0)}\chi(f)^p(0,0,0).
\end{align*}
On the other hand, from the multiplication of $\frak{g}$ and Remark \ref{200823907}, we have the following equations:
\begin{align*}
\sigma([H_i,H_i])&=2\sigma(h_i)=0, \quad \sigma([e,f])=\sigma(h_1-h_2)=0,\\
 \sigma([H_i, F])&=\sigma(f)=\delta_{\lambda_1\neq 0}\sigma(f)^{(\lambda_1,1,1)}(\lambda_1,1,1)+\sigma(f)^{(\lambda_1+1,0,0)}(\lambda_1+1,0,0),\\
\sigma([e,H_i])&=(-1)^i\sigma(E)=(-1)^i\delta_{\lambda_1\neq 0}\sigma(E)^{(\lambda_1-1,0,1)}(\lambda_1-1,0,1)\\
&\quad+(-1)^i\sigma(E)^{(\lambda_1-2,1,0)}(\lambda_1-2,1,0),\\
\sigma([f,E])&=\sigma(H_2-H_1)=\delta_{\lambda_1\neq 0}\left(\sigma(H_2)^{(\lambda_1,0,1)}-\sigma(H_1)^{(\lambda_1,0,1)}\right)(\lambda_1,0,1)\\
 &\quad+\left(\sigma(H_2)^{(\lambda_1-1,1,0)}-\sigma(H_1)^{(\lambda_1-1,1,0)}\right)(\lambda_1-1,1,0),
\end{align*}
where the first two equations are true from Lemma \ref{1649} (2).
Then each even weight-derivation is in the right set of the above claim.

Let $\sigma\in \mathrm{Der}(\frak{g},  Z_{\chi}(\lambda))_{(0),\bar{1}}$.
On one hand, by the definition of derivations and Lemma \ref{weightspace}, we get the following equations:
\begin{align*}
-\sigma([H_i, F])=&\delta_{\lambda_1\neq 0}\sigma(F)^{(\lambda_1,1,1)}H_i(\lambda_1,1,1)+\sigma(H_i)^{(\lambda_1,0,0)}F(\lambda_1,0,0)\\
&+\delta_{\lambda_1\neq 0}\sigma(H_i)^{(\lambda_1-1,1,1)}F(\lambda_1-1,1,1)\\
\stackrel{(\ref{F}),(\ref{a1k})}{=}&\delta_{\lambda_1\neq 0}\left(-\lambda_1\sigma(F)^{(\lambda_1,1,1)}+\sigma(H_i)^{(\lambda_1,0,0)}\right)(\lambda_1,1,0)\\
&+\delta_{\lambda_1\neq 0}\delta_{\lambda_1\neq p-1}\sigma(F)^{(\lambda_1,1,1)}(\lambda_1+1,0,1)\\
&+\delta_{\lambda_1=0}\sigma(H_i)^{(0,0,0)}(0,1,0)+\delta_{\lambda_1=p-1}\sigma(F)^{(p-1,1,1)}\chi(f)^p(0,0,1),
\end{align*}
\begin{align*}
\sigma([e,H_i])=&\delta_{\lambda_1\neq 0}\left(\sigma(H_i)^{(\lambda_1-1,1,1)}e(\lambda_1-1,1,1)+\sigma(e)^{(\lambda_1-1,0,1)}H_i(\lambda_1-1,0,1)\right)\\
&+\sigma(H_i)^{(\lambda_1,0,0)}e(\lambda_1,0,0)+\sigma(e)^{(\lambda_1-2,1,0)}H_i(\lambda_1-2,1,0)\\
\stackrel{(\ref{e})-(\ref{a0k}),(\ref{e0}),(\ref{a10})}{=}&\delta_{\lambda_1\neq 0}\left(\sigma(H_i)^{(\lambda_1-1,1,1)}\lambda_1(\lambda_1-1)
+(-1)^i\sigma(e)^{(\lambda_1-2,1,0)}\right)(\lambda_1-2,1,1)   \\
&+\delta_{\lambda_1\neq 0}(-1)^i(\lambda_1-1)\sigma(e)^{(\lambda_1-1,0,1)}(\lambda_1-2,1,1) \\
&+\delta_{\lambda_1\neq 0}\left(\sigma(e)^{(\lambda_1-1,0,1)}\lambda_1+\sigma(H_i)^{(\lambda_1,0,0)}\lambda_1(\lambda_1+1)\right)(\lambda_1-1,0,0)\\
&+\delta_{\lambda_1\neq 0}\delta_{\lambda_1\neq 1}\sigma(e)^{(\lambda_1-2,1,0)}(\lambda_1-1,0,0)\\
&+\delta_{\lambda_1= 1}\sigma(e)^{(p-1,1,0)}\chi(f)^p(0,0,0)+\delta_{\lambda_1= 0}\sigma(e)^{(p-2,1,0)}(p-1,0,0),
\end{align*}
\begin{align*}
\sigma([e,f])=&\delta_{\lambda_1\neq 0}\left(\sigma(f)^{(\lambda_1+1,0,1)}e(\lambda_1+1,0,1)-\sigma(e)^{(\lambda_1-1,0,1)}f(\lambda_1-1,0,1)\right)\\
&+\sigma(f)^{(\lambda_1,1,0)}e(\lambda_1,1,0)-\sigma(e)^{(\lambda_1-2,1,0)}f(\lambda_1-2,1,0)\\
\stackrel{(\ref{f}),(\ref{e}),(\ref{e0})}{=}&\delta_{\lambda_1\neq 0}\left(\sigma(f)^{(\lambda_1+1,0,1)}\lambda_1(\lambda_1+1)+2\sigma(f)^{(\lambda_1,1,0)}-\sigma(e)^{(\lambda_1-1,0,1)}\right)(\lambda_1,0,1)\\
&-\delta_{\lambda_1\neq 1}\sigma(e)^{(\lambda_1-2,1,0)}(\lambda_1-1,1,0)-\delta_{\lambda_1= 1}\sigma(e)^{(p-1,1,0)}\chi(f)^p(0,1,0)\\
&+\delta_{\lambda_1\neq 0}\sigma(f)^{(\lambda_1,1,0)}\lambda_1(\lambda_1-1)(\lambda_1-1,1,0),
\end{align*}
\begin{align*}
\sigma([f,H_i])=&\delta_{\lambda_1\neq 0}\left(\sigma(H_i)^{(\lambda_1-1,1,1)}f(\lambda_1-1,1,1)+\sigma(f)^{(\lambda_1+1,0,1)}H_i(\lambda_1+-1,0,1)\right)\\
&+\sigma(H_i)^{(\lambda_1,0,0)}f(\lambda_1,0,0)+\sigma(f)^{(\lambda_1,1,0)}H_i(\lambda_1,1,0)\\
\stackrel{(\ref{f})-(\ref{a0k}),(\ref{a10})}{=}&\delta_{\lambda_1\neq 0}\left(
(\lambda_1+1)\sigma(f)^{(\lambda_1+1,0,1)}+\sigma(f)^{(\lambda_1,1,0)}\right)(-1)^i(\lambda_1,1,1)   \\
&+\delta_{\lambda_1\neq 0}\lambda_1\sigma(f)^{(\lambda_1+1,0,1)}(\lambda_1+1,0,0)
+\delta_{\lambda_1\neq 0}\sigma(H_i)^{(\lambda_1-1,1,1)}(\lambda_1,1,1) \\
&+\delta_{\lambda_1\neq p-1}\left(\sigma(f)^{(\lambda_1,1,0)}+\sigma(H_i)^{(\lambda_1,0,0)}\right)(\lambda_1+1,0,0)\\
&+\delta_{\lambda_1=p-1}\left(\sigma(H_i)^{(p-1,0,0)}+\sigma(f)^{(p-1,1,0)}\right)\chi(f)^p(0,0,0),
\end{align*}
where the first equation is from Lemma \ref{1649} (3).
On the other hand, from Lemmas \ref{weightspace} and \ref{1649} (2), we have the following equations:
\begin{align*}
\sigma([e,f])&=\sigma(h_1-h_2)=0,\\
\sigma([e,H_i])&=(-1)^i\sigma(E)=(-1)^i\delta_{\lambda_1\neq 0}\sigma(E)^{(\lambda_1-2,1,1)}(\lambda_1-2,1,1),\\
\sigma([f,H_i])&=(-1)^{i+1}\sigma(F)=(-1)^{i+1}\delta_{\lambda_1\neq 0}\sigma(F)^{(\lambda_1,1,1)}(\lambda_1,1,1),\\
 \sigma([H_i, F])&=\sigma(f)=\delta_{\lambda_1\neq 0}\sigma(f)^{(\lambda_1+1,0,1)}(\lambda_1+1,0,1)+\sigma(f)^{(\lambda_1,1,0)}(\lambda_1,1,0),
\end{align*}
where the first three equations are true from Lemma \ref{1649}.
Then each even weight-derivation is in the right set of the above claim.

It follows that the above claim is true. As a result Formula (\ref{1Z}) holds.

For convenience, define some linear maps from $\frak{g}$ to $L_{\chi}(\lambda)$ as follows:
\begin{itemize}
\item[(1)] For $i=1,2$, define $\varphi_i\in \mathrm{Hom}_{\mathbb{F}}(\frak{g}, L_{\chi}(\lambda))$ by
$$\varphi_i(h_i)=(0,0,0),\quad \varphi_i(x)=0,$$
where  $x=e, E, F, f, H_1, H_2, \delta_{i=1}h_2, \delta_{i=2}h_1.$

\item[(2)] For $i=1,2$, define $\psi_i\in \mathrm{Hom}_{\mathbb{F}}(\frak{g}, L_{\chi}(\lambda))$ by
$$\psi_i(H_i)=(0,0,0),\quad \psi_i(x)=0, $$
where $x=e, E, F, f, h_1, h_2, \delta_{i=1}H_2, \delta_{i=2}H_1.$

\item[(3)] Define $\varphi_3\in \mathrm{Hom}_{\mathbb{F}}(\frak{g}, L_{\chi}(\lambda))$ by
$$\varphi_3(f)=-(0,0,0),\quad\varphi_3(F)=(0,0,1),\quad \varphi_3(x)=0,$$
where $x=e, E, h_1, h_2, H_1, H_2.$

\item[(4)] Define $\varphi_4\in \mathrm{Hom}_{\mathbb{F}}(\frak{g}, L_{\chi}(\lambda))$ by
$$\varphi_4(e)=(p-2,0,0),\quad\varphi_4(E)=(p-3,1,0),\quad \varphi_4(x)=0,$$
where $x=f, F, h_1, h_2, H_1, H_2.$

\item[(5)] Define $\psi_3\in \mathrm{Hom}_{\mathbb{F}}(\frak{g}, L_{\chi}(\lambda))$ by
$$\psi_3(E)=-2(0,0,0),\quad\psi_3(F)=(2,0,0),$$
$$\psi_3(H_1)=(1,0,0),\quad\psi_3(H_2)=-(1,0,0),\quad \psi_3(x)=0,$$
where $x=e, f, h_1, h_2.$
\end{itemize}

In the following, we give a proof of Theorem \ref{Th2}.

\textit{The proof of Theorem \ref{Th2}}:
Claim that
$$
\mathrm{Der}(\frak{g},L_{\chi}(\lambda))_{(0)}=\left\{\begin{array}{ll}
\langle \frak{D}_{(\lambda_1+1,0,0)}\mid\frak{D}_{(\lambda_1-1,1,0)}\rangle, &\chi=0, \lambda=(\lambda_1,-\lambda_1)\in\mathbb{F}_p^2\;\mbox{with}\;\lambda_1\neq 0,\pm1\\
\langle \varphi_1,\varphi_2\mid \psi_1,\psi_2\rangle, &(\lambda,\chi)=(0,0)\\
\langle  \frak{D}_{(1,0,0)}\mid  \frak{D}_{(0,1,0)},\psi_3\rangle, &\lambda=(1,p-1),\chi=0\\
\langle \varphi_3,\varphi_4\mid 0\rangle, &\lambda=(p-1,1),\chi=0\\
\langle \frak{D}_{(\lambda_1,0,0)}\mid \frak{D}_{(\lambda_1-1,1,0)}\rangle, &\chi\;\mbox{is nilpotent}\\
0, & \mbox{otherwise}.
\end{array}\right.
$$

It is routine to show that $\varphi_i, \psi_j$ are weight-derivations under the corresponding condition
for $1\leq i\leq 4, 1\leq j\leq 3$.
It is a standard fact that the space
$\mathrm{Ider}(\frak{g},L_{\chi}(\lambda))_{(0)}=0$ in case $(\lambda,\chi)=(0,0)$ or $\lambda=(p-1,1),\chi=0$,
which implies that  the derivations $\varphi_i, \psi_j$ are not inner, where $i=1,2,3,4, j=1,2$.
In addition, $\mathrm{Ider}(\frak{g},L_{\chi}(\lambda))_{(0),\bar{1}}=\mathbb{F}\frak{D}_{(0,1,0)}$ in case $\lambda=(1,p-1),\chi=0$,
which implies that the derivation $\psi_3$ is also not inner.
It is true that the elements in the right sets of the above claim are linearly independent,
the proof of which is omitted.

In the following, we prove that weight-derivations must be in the right sets of the above claim.

For convenience, denote by $P_2$ the proposition that $1\leq\lambda_1\leq \frac{p-1}{2}$;
$P_3$ the one that $1\leq\lambda_1\leq \frac{p-1}{2}$ or $\lambda_1=p-1$;
$P_5$ the one that $2\leq\lambda_1\leq \frac{p-1}{2}$ or $\lambda_1=p-1$.

Let $\sigma\in \mathrm{Der}(\frak{g}, L_{\chi}(\lambda))_{(0),\bar{0}}$.
On one hand, by the definition of derivations and Lemma \ref{weightspacel}, we get the following equations:
\begin{align*}
\sigma([H_i, F])=&\sigma(F)^{(\lambda_1,1,0)}\left(\delta_{\chi=0}\delta_{P_2}
+\delta_{P_1}\right)H_i(\lambda_1,1,0)\\
&+\sigma(H_i)^{(\lambda_1-1,1,0)}\left(\delta_{\chi=0}\delta_{P_2}
+\delta_{P_1}\right)F(\lambda_1-1,1,0)\\
&+\delta_{(\chi,\lambda_1)=(0,p-1)}\sigma(F)^{(0,0,1)}H_i(0,0,1)\\
\stackrel{(\ref{001})-(\ref{a10l})}{=}&\sigma(F)^{(\lambda_1,1,0)}\left(\delta_{\chi=0}\delta_{P_2}
+\delta_{P_1}\delta_{\lambda_1\neq p-1}\right)(\lambda_1+1,0,0)\\
&+\sigma(F)^{(\lambda_1,1,0)}\delta_{P_1}\delta_{\lambda_1= p-1}\chi(f)^p(0,0,0)\\
&-\delta_{(\chi,\lambda_1)=(0,p-1)}\sigma(F)^{(0,0,1)}(0,0,0),
\end{align*}
\begin{align*}
\sigma([H_i, E])=&\sigma(E)^{(\lambda_1-2,1,0)}\left(\delta_{\chi=0}\delta_{P_5}
+\delta_{P_1}\right)H_i(\lambda_1-2,1,0)\\
&+\sigma(H_i)^{(\lambda_1-1,1,0)}\left(\delta_{\chi=0}\delta_{P_2}
+\delta_{P_1}\right)E(\lambda_1-1,1,0)\\
&+\delta_{(\chi,\lambda_1)=(0,1)}\sigma(E)^{(0,0,1)}H_i(0,0,1)\\
\stackrel{(\ref{001}),(\ref{a10l}),(\ref{El})}{=}&\sigma(E)^{(\lambda_1-2,1,0)}\left(\delta_{\chi=0}\delta_{P_5}
+\delta_{P_1}\delta_{\lambda_1\neq 1}\right)(\lambda_1-1,0,0)\\
&+\sigma(E)^{(\lambda_1-2,1,0)}\delta_{P_1}\delta_{\lambda_1= 1}\chi(f)^p(0,0,0)\\
&+\delta_{(\chi,\lambda_1)=(0,1)}\sigma(E)^{(0,0,1)}(0,0,0),
\end{align*}
\begin{align*}
\sigma([e,H_i])=&\sigma(H_i)^{(\lambda_1-1,1,0)}\left(\delta_{\chi=0}\delta_{P_2}
+\delta_{P_1}\right)e(\lambda_1-1,1,0)\\
&-\sigma(e)^{(\lambda_1-1,0,0)}\left(\delta_{\chi=0}\delta_{P_3}+\delta_{P_1}\right)H_i(\lambda_1-1,0,0)\\
\stackrel{(\ref{a00l}),(\ref{el})}{=}&\lambda_1(\lambda_1+1)\sigma(H_i)^{(\lambda_1-1,1,0)}\left(\delta_{\chi=0}\delta_{P_4}+\delta_{P_1}\right)(\lambda_1-2,1,0)\\
&+\delta_{(\chi,\lambda_1)=(0,1)}\left(2\sigma(H_i)^{(\lambda_1-1,1,0)}+(-1)^i\sigma(e)^{(\lambda_1-1,0,0)}\right)(0,0,1)\\
&+(-1)^i\sigma(e)^{(\lambda_1-1,0,0)}\left(\delta_{\chi=0}\delta_{P_5}+\delta_{P_1}\right)(\lambda_1-2,1,0),
\end{align*}
\begin{align*}
\frac{1}{2}\sigma([H_i,H_i])=&\sigma(H_i)^{(\lambda_1-1,1,0)}\left(\delta_{\chi=0}\delta_{P_2}
+\delta_{P_1}\right)H_i(\lambda_1-1,1,0)\\
\stackrel{(\ref{a10l})}{=}&\sigma(H_i)^{(\lambda_1-1,1,0)}\left(\delta_{\chi=0}\delta_{P_2}
+\delta_{P_1}\right)(\lambda_1,0,0).
\end{align*}
On the other hand, from the multiplication of $\frak{g}$, Lemmas \ref{weightspacel} and \ref{1649},
we have the following equations:
\begin{align*}
\sigma([H_i,H_i])&=2\sigma(h_i)=0,\\
 \sigma([H_i, F])&=\sigma(f)=\sigma(f)^{(\lambda_1+1,0,1)}\left(\delta_{\chi=0}\delta_{P_3}+\delta_{P_1}\right)(\lambda_1+1,0,1),\\
 \sigma([H_i, E])&=\sigma(e)=\sigma(e)^{(\lambda_1-1,0,0)}\left(\delta_{\chi=0}\delta_{P_3}+\delta_{P_1}\right)(\lambda_1-1,0,0),\\
\sigma([e,H_i])&=(-1)^i\sigma(E)=(-1)^i\left(\delta_{\chi=0}\delta_{P_5}+\delta_{P_1}\right)\sigma(E)^{(\lambda_1-2,1,0)}(\lambda_1-2,1,0)\\
&\quad\quad\quad\quad\quad\quad\quad+(-1)^i\delta_{(\chi,\lambda_1)=(0,1)}\sigma(E)^{(0,0,1)}(0,0,1).
\end{align*}
Then each even weight-derivation is in the right sets of the above claim.

Let $\tau\in \mathrm{Der}(\frak{g}, L_{\chi}(\lambda))_{(0),\bar{1}}$.
On one hand, by the definition of derivations and Lemma \ref{weightspacel}, we get the following equations:
\begin{align*}
\tau([e,H_i])=&\tau(e)^{(\lambda_1-2,1,0)}\left(\delta_{\chi=0}\delta_{P_5}
+\delta_{P_1}\right)H_i(\lambda_1-2,1,0)\\
&+\tau(H_i)^{(\lambda_1,0,,0)}\left(\delta_{\chi=0}\delta_{P_2}
+\delta_{P_1}\right)e(\lambda_1,0,0)\\
&+\delta_{(\chi,\lambda_1)=(0,1)}\tau(e)^{(0,0,1)}H_i(0,0,1)\\
\stackrel{(\ref{001}),(\ref{a10l}),(\ref{el})}{=}&\tau(H_i)^{(\lambda_1,0,0)}\lambda_1(\lambda_1+1)\left(\delta_{\chi=0}\delta_{P_2}
+\delta_{P_1}\right)(\lambda_1-1,0,0)\\
&+\tau(e)^{(\lambda_1-2,1,0)}\left(\delta_{\chi=0}\delta_{P_5}
+\delta_{P_1}\delta_{\lambda_1\neq 1}\right)(\lambda_1-1,0,0)\\
&+\tau(e)^{(\lambda_1-2,1,0)}\left(\delta_{(\chi,\lambda_1)=(0,1)}
+\delta_{P_1}\delta_{\lambda_1= 1}\right)(0,0,0),
\end{align*}
\begin{align*}
\tau([f,H_i])=&\tau(f)^{(\lambda_1,1,0)}\left(\delta_{\chi=0}\delta_{P_2}
+\delta_{P_1}\right)H_i(\lambda_1,1,0)\\
&+\tau(H_i)^{(\lambda_1,0,,0)}\left(\delta_{\chi=0}\delta_{P_2}
+\delta_{P_1}\right)f(\lambda_1,0,0)\\
&+\delta_{(\chi,\lambda_1)=(0,p-1)}\tau(f)^{(0,0,1)}H_i(0,0,1)\\
\stackrel{(\ref{001}),(\ref{a10l}),(\ref{fl})}{=}&\tau(H_i)^{(\lambda_1,0,0)}\left(\delta_{\chi=0}\delta_{P_2}
+\delta_{P_1}\delta_{\lambda_1\neq p-1}\right)(\lambda_1+1,0,0)\\
&+\tau(f)^{(\lambda_1,1,0)}\left(\delta_{\chi=0}\delta_{P_2}
+\delta_{P_1}\delta_{\lambda_1\neq p-1}\right)(\lambda_1+1,0,0)\\
&+\tau(f)^{(\lambda_1,1,0)}\left(\delta_{P_1}\delta_{\lambda_1=p-1}\chi(f)^p-\delta_{(\chi,\lambda_1)=(0,p-1)}\right)(0,0,0)\\
&+\tau(H_i)^{(\lambda_1,0,0)}\delta_{P_1}\delta_{\lambda_1= p-1}\chi(f)^p(0,0,0).
\end{align*}
On the other hand, from Lemma \ref{1649}, we have the following equations:
\begin{align*}
\tau([e,H_i])&=(-1)^i\tau(E)=(-1)^i\delta_{(\chi,\lambda_1)=(0,1)}\tau(E)^{(0,0,0)}(0,0,0),\\
\tau([f,H_i])&=(-1)^{i+1}\tau(F)=(-1)^{i+1}\delta_{(\chi,\lambda_1)=(0,1)}\tau(F)^{(2,0,0)}(2,0,0).
\end{align*}
Then each odd weight-derivation is in the right sets of the above claim.

It follows that the above claim is true. As a result Theorem \ref{Th2} holds.

\textbf{Statements and Declarations:}
All data generated or analysed during this study are included in this manuscript.

\end{document}